\documentclass[11pt]{report}
\usepackage{amscd,amsmath,amsthm,amssymb}
\usepackage[all]{xypic}

\theoremstyle{plain}  

\newtheorem{thm}{Theorem}[section]
\newtheorem{prop}[thm]{Proposition}
\newtheorem{lem}[thm]{Lemma}   
\newtheorem{cor}[thm]{Corollary}

\theoremstyle{definition}

\newtheorem{prop-defn}[thm]{Proposition--Definition}
\newtheorem{defn}[thm]{Definition}

\newtheorem{say}[thm]{}

\theoremstyle{remark}

\newtheorem{claim}{Claim}
\newtheorem{rem}[thm]{Remark}
\newtheorem{ex}[thm]{Example}

\DeclareMathOperator{\Mob}{Mob}
\DeclareMathOperator{\Div}{Div}

\DeclareMathOperator{\Adj}{Adj}
\DeclareMathOperator{\img}{Im}
\DeclareMathOperator{\divr}{div}

\DeclareMathOperator{\Supp}{Supp}
\DeclareMathOperator{\cod}{cod}

\newcommand{\QED}{\ifhmode\unskip\nobreak\fi\quad {\rm Q.E.D.}} 

\newcommand{\bA}{\mathbb A}
\newcommand{\bC}{\mathbb C}

\newcommand{\bN}{\mathbb N}

\newcommand{\bQ}{\mathbb Q}
\newcommand{\bR}{\mathbb R}

\newcommand{\bZ}{\mathbb Z}

\newcommand{\cF}{\mathcal F}

\newcommand{\cJ}{\mathcal J}
\newcommand{\cO}{\mathcal O}

\newcommand{\sR}{\mathcal R}

\newcommand{\nsubset}{\not\!\subset}

\renewcommand{\o}[0]{{\mathcal O}} 

\renewcommand{\r}[0]{{\mathbb R}}

\newcommand{\q}[0]{{\mathbb Q}}
\newcommand{\map}[0]{\dasharrow}
\newcommand{\qtq}[1]{\quad\mbox{#1}\quad}

\newcommand{\pic}[0]{\operatorname{Pic}}

\newcommand{\supp}[0]{\operatorname{Supp}}

\newcommand{\cent}[0]{\operatorname{center}}

\newcommand{\exc}[0]{\operatorname{Ex}}

\newcommand{\rdown}[1]{\lfloor{#1}\rfloor}

\newcommand{\locus}[0]{\operatorname{locus}}

\title{Lectures on Flips and Minimal Models}

\author{Alessio Corti, Paul Hacking, J\'anos Koll\'ar, Robert Lazarsfeld,\\ and Mircea Musta\c{t}\u{a}}

\begin{document}
\maketitle

This document contains notes from the lectures of Corti, Koll\'ar, Lazarsfeld, and Musta\c{t}\u{a} at the workshop ``Minimal and canonical models in algebraic geometry" at the Mathematical Sciences Research Institute, Berkeley, CA, April 16-20, 2007. The lectures give an overview of the recent advances
on canonical and minimal models of algebraic varieties obtained in the papers \cite{HM}, \cite{BCHM}. 

Lecture~1, by Lazarsfeld, explains the lifting theorems used in the proof.
Lectures~2 and 3 describe the proof of the existence of flips in dimension $n$ assuming the MMP in dimension $n-1$ from \cite{HM}.
Lecture~2, by Musta\c{t}\u{a}, gives a geometric description of the restriction of the log canonical algebra to a boundary divisor
(it is a so called adjoint algebra). Lecture~3, by Corti, proves that this algebra is finitely generated, following ideas of Shokurov.
Lecture~4, by Koll\'{a}r, gives an overview of the paper \cite{BCHM} on the existence of minimal models for varieties of log general type.

Videos of the lectures are available at 

\medskip
\noindent
www.msri.org/calendar/workshops/WorkshopInfo/418/show\_workshop

\bigskip
\noindent
{\bf Acknowledgements}.
We thank MSRI for hosting the conference and financial support.
Hacking, Koll\'ar, Lazarsfeld, and Musta\c{t}\u{a} were partially supported by NSF grants DMS-0650052, DMS-0500198, DMS-0139713, and DMS-0500127, respectively.

\chapter{Extension theorems} \label{lazarsfeld}

\begin{thm}(Hacon--McKernan)
Let $X$ be a smooth projective variety over $\bC$.
Let $T+\Delta$ be an effective $\bQ$-divisor with snc support, with $T$ irreducible and $\lfloor \Delta \rfloor = 0$.
Assume $\Delta \sim_{\bQ} A+B$, where $A$ is ample, $B$ is effective, and $T \nsubset \Supp B$. 
Assume that the stable base locus of $K_X+T+\Delta$ does not contain any intersection of irreducible components of $T+\Delta$.
Choose $k$ such that $k\Delta$ is integral and set $L=k(K_X+T+\Delta)$.
Then the restriction map
$$H^0(X,mL) \rightarrow H^0(T,mL_T)$$
is  surjective for all $m \ge 1$.
\end{thm}
Here and in what follows we write $D_T$ for the restriction of a divisor $D$ to a subvariety $T$.

\begin{rem}
Hacon and McKernan prove an analogous result for a projective morphism $X \rightarrow Z$ with $Z$ affine.
\end{rem}

History: The proof uses an idea of Siu which first appeared in his work on deformation invariance of plurigenera.
There are also related works by Kawamata, Takayama, and others.

\section{Multiplier and adjoint ideals}

For an effective $\bQ$-divisor $D$ on $X$, we have the multiplier ideal $\cJ(X,D) \subseteq \cO_X$.
Roughly speaking it ``measures" the singularities of the pair $(X,D)$ --- worse singularities correspond to
deeper ideals.

Let $L$ be a big line bundle on $X$. Set $$\cJ(X,\|L\|)=\cJ(X,\frac{1}{p}D_p)$$ where $D_p \in |pL|$ is general and $p \gg 0$.

\begin{prop}
\begin{enumerate}
\item $H^0(L\otimes\cJ(X,\|L\|)) \stackrel{\sim}{\longrightarrow} H^0(L)$, i.e., every section of $L$ vanishes on $\cJ(\|L\|)$.
\item If $M=K_X+L+P$ where $P$ is nef then $$H^i(X,M\otimes\cJ(\|L\|))=0$$ for all $i>0$ (this is a restatement of Kawamata--Viehweg--Nadel vanishing).
\item If $M=K_X+L+(\dim X+1)H$ where $H$ is very ample then $M \otimes \cJ(\|L\|)$ is globally generated (this follows from (2) plus Castelnuovo--Mumford regularity: if a sheaf $\cF$ on projective space satisfies $H^i(\cF(-i))=0$ for all $i>0$ then $\cF$ is globally generated).
\end{enumerate}
\end{prop}
Here, given an ideal $\mathfrak{a} \subseteq \cO_X$, line bundle $M$, and section $s \in \Gamma(X,M)$, we say $s$ \emph{vanishes on} $\mathfrak{a}$ if
$s \in \img(\Gamma(M \otimes \mathfrak{a}) \rightarrow \Gamma(M))$.

Consider a smooth divisor $T \subset X$, $T \nsubset \Supp D$.
We can define the \emph{adjoint ideal} $\Adj_T(X,D) \subseteq \cO_X$, such that
$$0 \rightarrow \cJ(X,D) \otimes \cO_X(-T) \rightarrow \Adj \rightarrow \cJ(T,D_T) \rightarrow 0.$$
Assume that $T$ is not contained in the stable base locus of $L$.
We have
$$\cJ(T, \|L\|_T) := \cJ(T,\frac{1}{p}D_p|_T)  \subseteq \cJ(T,\|L_T\|)$$
where $D_p \in |pL|$ is general and $p \gg 0$.
We get $\Adj_T(X,\|L\|) \subseteq \cO_X$, with
$$0 \rightarrow \cJ(X,\|L\|) \otimes \cO_X(-T) \rightarrow \Adj \rightarrow \cJ(T,\|L\|_T) \rightarrow 0.$$
Now apply $(\cdot)\otimes M$.
Idea: If $s \in \Gamma(T,M_T)$ vanishes on $\cJ(T,\|L\|_T)$ and if $M-(K_X+L+T)$ is nef, then $s$ extends to a section of $\cO_X(M)$
(we get vanishing of $H^1$ by Kawamata--Viehweg vanishing).

Consider as above $L=k(K_X+T+\Delta)$, $\Delta=A+B$, etc.

\begin{lem}(Main Lemma) There exists a very ample divisor $H$ (independent of $p$) such that for every $p \ge 0$, every 
$\sigma \in \Gamma(T, \cO(pL_T+H_T)\otimes\cJ(\|pL_T\|))$ extends to $\hat{\sigma} \in \Gamma(X,\cO_X(pL+H))$.
\end{lem}

Assuming the lemma, we prove the theorem:

(0) We may assume $(T,B_T)$ is klt. Take $h \in \Gamma(X,H)$ general.

(1) Let $s \in \Gamma(T,mL_T)$. Consider $\sigma=s^l\cdot h \in \Gamma(T,lmL_T+H_T)$.
The section $\sigma$ vanishes on $\cJ(T,\|lmL_T\|)$. So there exists $\hat{\sigma} \in \Gamma(X,lmL+H)$ such that
$\hat{\sigma}|_T=\sigma$ (by the Main Lemma).

(2) Let $F=\frac{mk-1}{mlk}\divr(\hat{\sigma})+B$. We find (using $\Delta=A+B$) that $mL-F-T=K_X+(\mbox{ample})$ if $l \gg 0$.

(3) Using $(T,B_T)$ klt, we check that $\cO_T(-\divr(s)) \subseteq \cJ(T,F_T)$. So $s$ vanishes on $\cJ(T,F_T)$.

(4) Finally, consider the sequence 
$$0 \rightarrow \cJ(X,F) \otimes \cO_X(-T) \rightarrow \Adj \rightarrow \cJ(T,F_T) \rightarrow 0$$
tensored by $mL$.
We have $s \in H^0(\cO_T(mL_T)\otimes\cJ(T,F_T))$ and $H^1(\cO_X(mL-T)\otimes\cJ(X,F))=0$ by vanishing. So $s$ extends.

\section{Proof of the Main Lemma}

We will only prove the special case $k=1$, $L=K_X+T$ (so $\Delta=0$ --- we don't need $\Delta=A+B$ here),
$T$ not contained in the stable base locus of $L$. (Note: $L_T=K_T$).
We prove

$(*)_p$ : If $H=(\dim X +1)(\mbox{very ample})$, then $$H^0(T,\cO(pL_T+H_T)\otimes \cJ(\|(p-1)L_T\|))$$ lifts to $H^0(X,pL+H)$.

The proof is by induction on $p$. The case $p=1$ is OK by vanishing. Assume $(*)_p$ holds.

\begin{claim}
$\cJ(T,\|(p-1)L_T\|) \subseteq \cJ(T,\|pL+H\|_T)$.
\end{claim}
\begin{proof}
By Castelnuovo--Mumford regularity and vanishing 
$$\cO_T(pL_T+H_T)\otimes\cJ(\|(p-1)L_T\|) \qquad (**)$$ 
is globally generated.

$(*)_p$ implies sections of $(**)$ lift to $X$. So

$$\cJ(T,\|(p-1)L_T\|) \subseteq b(X,|pL+H|)\cdot\cO_T \subseteq \cJ(T,\|pL+H\|_T).$$
(here $b(X,|D|)$ denotes the ideal defining the base locus of $|D|$).
\end{proof}

Consider the adjoint sequence

$$0 \rightarrow \cJ(\|pL+H\|)(-T) \rightarrow \Adj \rightarrow \cJ(T,\|pL+H\|_T) \rightarrow 0.$$

Apply $(\cdot) \otimes \cO_X((p+1)L+H)$. By the Claim 
$$H^0(\cO_T((p+1)L_T+H_T) \otimes\cJ(\|pL_T\|)) \subseteq H^0(\cO_T((p+1)L_T+H_T) \otimes \cJ(T,\|pL+H\|_T))$$
and $$H^1(\cO_X((p+1)L+H-T) \otimes \cJ(\|pL+H\|))=0.$$ This gives the desired lifting.

\chapter{Existence of flips I} \label{mustata}

This chapter is (mostly) an exposition of work of Hacon and McKernan.

\section{The setup}

Let $f \colon (X,D) \rightarrow Z$ be a birational projective
morphism, where $X$ is $\bQ$-factorial, $D$ is an effective
$\bQ$-divisor, $Z$ is normal, $\rho(X/Z)=1$, and $-(K_X+D)$ is
$f$-ample. Assume also that $f$ is small, i.e., that the exceptional
locus has codimension $\ge 2$. We consider two cases:
\begin{enumerate}
\item klt flip: $(X,D)$ klt.
\item pl flip: $(X,D)$ plt, $D=S+\Delta$ with $\lfloor D \rfloor = S$ irreducible, and $-S$ is $f$-ample.
\end{enumerate}

It is well-known that the flip of $f$ exists iff the $\cO_Z$-algebra
$$\oplus_{m \ge 0} f_*\cO_X(\lfloor m(K_X+D)\rfloor)$$
is finitely generated. This is a local question on $Z$, hence we may
and will assume that $Z$ is affine.

Shokurov: MMP in dimension $n-1$ plus existence of pl flips in
dimension $n$ implies existence of klt flips in dimension $n$.
Therefore we need only consider the case of pl flips.

\begin{rem}\label{rem1}
Let $R= \oplus_{i \ge 0} R_i$ be a graded domain such that $R_0$ is a finitely generated $\bC$-algebra.
Then the algebra $R$ is finitely generated iff the \emph{truncation}
$$R_{(k)}:= \oplus_{i \ge 0} R_{ki}$$
is finitely generated. Indeed, we have an obvious action of
${\mathbb Z}/k {\mathbb Z}$ on $R$ such that the ring of invariants
is $R_{(k)}$, hence $R$ finitely generated implies $R_{(k)}$
finitely generated. To see the converse, it is enough to note that
for every $0<j<k$, if  $s\in \oplus_{i \ge 0} R_{ki+j}$ is a nonzero
homogeneous element, then multiplication by $s^{k-1}$ embeds
$\oplus_{i\geq 0} R_{ki+j}$ as an ideal of $R_{(k)}$.
\end{rem}

From now on, we assume that we are in the pl flip setting.

\begin{rem}\label{rem2}
Since $\rho(X/Z)=1$ and since we work locally over $Z$, it follows
from our assumption that we may assume that there are positive
integers $p$ and $q$ such that $p(K_X+D)\sim qS$ are linearly
equivalent Cartier divisors.
\end{rem}

\begin{rem}\label{rem3}
Since $Z$ is affine and $f$ is small, it follows that $S$ is
linearly equivalent to an effective divisor not containing $S$ in
its support. In particular, it follows from the previous remark that
there is a positive integer $k$ and $G\in |k(K_X+D)|$ such that
$S\nsubset {\rm Supp}(G)$.
\end{rem}

A key remark due to Shokurov is that the algebra
$$\sR= \oplus_{m\geq 0} H^0(X,\cO_X(\lfloor m(K_X+S+\Delta)\rfloor))$$ is finitely generated iff
the \emph{restricted algebra}
$$\sR|_S:= \oplus_{m \ge 0} \img\left(H^0(X,m(K_X+S+\Delta)) \rightarrow H^0(S,m(K_X+S+\Delta)|_S)\right)$$
is finitely generated. Sketch of proof: replacing $X$ by a suitable
$U \subset X$ such that $\cod(X-U) \ge 2$, we may assume that $S$ is
Cartier. Since
$$p(K_X+S+\Delta) \sim qS$$
for some positive integers $p$ and $q$, it follows from
Remark~\ref{rem1} that it is enough to show that the algebra
$\sR'=\oplus_{m\geq 0} H^0(X,\cO(mS))$ is finitely generated. Using
the fact that $\sR|_S$ is finitely generated and Remark~\ref{rem1},
we deduce that the quotient $\sR'/h\sR'$ is finitely generated,
where $h\in \sR'_1$ is an equation for $S$. Therefore $\sR'$ is
finitely generated.

\bigskip

The above discussion shows that existence of pl flips in dimension
$n$ (and so by Shokurov's result, existence of klt flips in
dimension $n$) follows from MMP in dimension $(n-1)$ and the
following Main Theorem, due to Hacon and McKernan.

\begin{thm}\label{main_thm}
Let $f \colon (X,S+\Delta) \rightarrow Z$ be a projective birational
morphism, where $X$ is $\bQ$-factorial, $S+\Delta$ is an effective
$\bQ$-divisor divisor with $\lfloor S+\Delta \rfloor=S$ irreducible,
and $Z$ is a normal affine variety.
 Let $k$ be a positive integer such that $k(K_X+S+\Delta)$ is Cartier.
Suppose
\begin{enumerate}
\item $(X,S+\Delta)$ is plt.
\item $S$ is not contained in the base locus of $|k(K_X+S+\Delta)|$.
\item $\Delta \sim_{\bQ} A+B$ where $A$ is ample, $B$ is effective, and $S \nsubset \Supp B$.
\item $-(K_X+S+\Delta)$ is $f$-ample.
\end{enumerate}
If the MMP (with $\bR$-coefficients) holds in dimension $\dim X -1$, then the restricted algebra
$$\sR|_S = \oplus_{m \ge 0} \img(H^0(X,L^m) \rightarrow H^0(S,L^m|_S))$$
is finitely generated, where $L=\cO_X(k(K_X+S+\Delta))$.
\end{thm}

\section{Adjoint algebras}

In the setting of Theorem~\ref{main_thm},  we have by adjunction
$(K_X+S+\Delta)|_S=K_S+\Delta|_S$. Moreover, the pair
$(S,\Delta|_S)$ is klt. The trouble comes from the fact that the
maps
\begin{equation}\label{eq11}
H^0(X,mk(K_X+S+\Delta)) \rightarrow H^0(S,mk(K_S+\Delta|_S))
\end{equation}
 are not
surjective in general.

Goal: find a model $T \rightarrow S$ such that (some truncation of) the restricted algebra $\sR|_S$ can be written as
$$\oplus_{m \ge 0} H^0(T,B_m),$$
where $\{ B_m \}$ is an additive sequence of Cartier divisors on $T$.

Terminology: an \emph{additive sequence} is a sequence of divisors
$\{ B_m \}$ on a normal variety $T$ such that $B_i+B_j \le B_{i+j}$
for all $i,j$. Note that in this case $\oplus_{m\geq 0}H^0(T,
\cO(B_m))$ has a natural algebra structure. A typical example of
additive sequence: start with a divisor $D$ such that $|D|$ is
nonempty, and let $B_m={\rm Mob}(mD) :=mD-{\rm Fix}|mD|$. More
generally, if $\{D_m\}$ is an additive sequence such that
$|D_m|\neq\emptyset$ for every $m$, and if we put $B_m={\rm
Mob}(D_m)$, then $\{B_m\}$ forms an additive sequence.

Given an additive sequence $\{B_m\}$, the \emph{associated convex
sequence} is given by $\{\frac{1}{m}B_m\}$. If each $\frac{1}{m}B_m$
is bounded above by a fixed divisor, set
$$B:= \sup \frac{1}{m}B_m = \lim_{m \rightarrow \infty} \frac{1}{m}B_m.$$

\begin{rem}
If $B$ is semiample and if there exists $i$ such that
$B=\frac{1}{i}B_i$ (hence, in particular, $B$ is a $\bQ$-divisor),
then $\oplus_{m \ge 0} H^0(T, \cO(B_m))$ is finitely generated.
Indeed, in this case $B=\frac{1}{m}B_m$ whenever $i\mid m$, and it
is enough to use the fact that if $L$ is a globally generated line
bundle, then $\oplus_{m\geq 0}H^0(X,L^m)$ is finitely generated.
\end{rem}

Suppose that $f \colon T \rightarrow Z$ is a projective morphism,
where $T$ is smooth and $Z$ is affine. An \emph{adjoint algebra} on
$T$ is an algebra of the form
$$\oplus_{m \ge 0} H^0(T,\cO(B_m)),$$ where $\{B_m\}$ is an additive sequence and
$B_m=mk(K_T+\Delta_m)$ for some $k \ge 1$ and $\Delta_m \ge 0$ such that
$\Delta := \lim_{m \rightarrow \infty} \Delta_m$ exists and $(T,\Delta)$ is klt.

Our goal in what follows is to show that under the hypothesis of
Theorem~\ref{main_thm} (without assuming $-(K_X+S+\Delta)$ ample or
MMP in $\dim(X)-1$), the algebra $\sR|_S$ can be written as an
adjoint algebra. It is shown in Chapter~\ref{corti} how one can
subsequently use the fact that $-(K_X+S+\Delta)$ is ample to deduce
that $\sR|_S$ is ``saturated", and then use MMP in $\dim(X)-1$ to
reduce to the case when the limit of the above $\Delta_m$ is such
that $K_T+({\rm this}\,{\rm  limit})$ is semiample. This is enough
to give the finite generation of $\sR|_S$ (Shokurov proved this
using diophantine approximation).

\section{The Hacon--McKernan extension theorem}

As we have already mentioned, the difficulty comes from the
non-surjectivity of the restriction maps (\ref{eq11}). We want to
replace $X$ by higher models on which we get surjectivity of the
corresponding maps as an application of the following Extension
Theorem, due also to Hacon and McKernan.

\begin{thm}\label{thm2}
Let $(Y,T+\Delta)$ be a pair with $Y$ smooth and $T+\Delta$ an
effective $\bQ$-divisor with snc support, with $\lfloor T+\Delta
\rfloor=T$ irreducible. Let $k$ be a positive integer such that
$k\Delta$ is integral and set $L=k(K_Y+T+\Delta)$. Suppose
\begin{enumerate}
\item $\Delta \sim_{\bQ} A+B$ where $A$ is ample, $B$ is effective, and $T \nsubset \Supp(B)$.
\item No intersection of components of $T+\Delta$ is contained in the base locus of $L$.
\end{enumerate}
Then the restriction map
$$H^0(Y,L) \rightarrow H^0(T,L|_T)$$
is surjective.
\end{thm}

This theorem was discussed in Chapter~\ref{lazarsfeld}.

\section{The restricted algebra as an adjoint algebra}

Let $f \colon (X,S+\Delta) \rightarrow Z$ be a projective morphism,
where $X$ is $\bQ$-factorial, $S+\Delta$ is an effective
$\bQ$-divisor with $\lfloor S+\Delta \rfloor=S$ irreducible,
$(X,S+\Delta)$ is plt, and $Z$ is an affine normal variety. Let $k$
be a positive integer such that $k(K_X+S+\Delta)$ is Cartier and $S$
is not contained in the base locus of $|k(K_X+S+\Delta)|$. Assume
$$ \Delta \sim_{\bQ} A+B \qquad (*)$$
where $A$ is ample, $B$ is effective, and $S \nsubset \Supp B$.

We will replace $(X,S+\Delta)$ by a log resolution (in fact, a
family of resolutions) on which we can apply the extension theorem
and use this to exhibit the restricted algebra as an adjoint
algebra. Consider a birational morphism $f \colon Y \rightarrow X$,
and let $T$ be the strict transform of $S$. Write
$$K_Y+T+\Delta_Y=f^*(K_X+S+\Delta)+E$$
where $\Delta_Y$ and $E$ are effective and have no common components, $f_* \Delta_Y = \Delta$, and $E$ is exceptional.

Since $E$ is effective and exceptional, we have
$H^0(X,mk(K_X+S+\Delta))\simeq H^0(Y,mk(K_Y+T+\Delta_Y))$ for every
$m$. Therefore we may ``replace" $K_X+S+\Delta$ by $K_Y+T+\Delta_Y$.

{\bf Step 1}. After replacing $\Delta$ by the linearly equivalent
divisor $A'+B'$, where $A'=\epsilon A$ and
$B'=(1-\epsilon)\Delta+\epsilon B$, with $0<\epsilon\ll 1$, we may
assume that we have equality of divisors $\Delta = A+B$ in $(*)$. We
may also assume (after possibly replacing $k$ by a multiple) that
$kA$ is very ample and that
$$A=\frac{1}{k} \cdot(\mbox{very general member of $|kA|$}).$$

Since $A$ is general in the above sense, if $f\colon Y\rightarrow X$
is a projective birational morphism, we may assume $f^*A=\tilde{A}$
is the strict transform. In this case, $(*)$ will also hold for
$(Y,T+\Delta_Y)$: indeed, there exists an effective exceptional
divisor $E'$ such that $f^*A -E'$ is ample, hence
$$\Delta_Y = (f^*A-E')+E'+ (\cdots),$$
where $(\cdots)$ is an effective divisor that does not involve $T$.

\emph{Caveat}: we will construct various morphisms $f$ as above
starting from $(X,S+\Delta)$. We will then modify $\Delta$ to
satisfy $f^*A=\tilde{A}$, and therefore we need to check how this
affects the properties of $f$.

{\bf Step 2}. Let  $f\colon Y \rightarrow X$ be a log resolution of
$(X,S+\Delta)$. After modifying $\Delta$ as explained in Step 1, we
write $\Delta_Y=\tilde{A}+\sum a_iD_i$ (note that $f$ remains a log
resolution for the new $\Delta$).

\begin{claim}
After blowing up intersections of the $D_i$ (and of their strict
transforms) we may assume $D_i \cap D_j = \emptyset$ for all $i \neq
j$.
\end{claim}
The proof of the claim is standard, by induction first on the number
of intersecting components and then on the sum of the coefficients
of intersecting components. Note also that as long as we blowup loci
that have snc with $\tilde{A}$, the condition $f^*(A)=\tilde{A}$ is
preserved.

{\bf Step 3}. We need to satisfy condition 2) in Theorem~\ref{thm2}.
The hypothesis implies that  $T\nsubset {\rm
Bs}|k(K_Y+T+\Delta_Y)|$. Since $A$ is general, it follows that we
need to worry only about the components $D_i$ or about the
intersections $D_i\cap T$ that are contained in ${\rm
Bs}|k(K_Y+T+\Delta_Y)|$.

Canceling common components, we may replace $\Delta_Y$ by $0 \le
\Delta'_Y \le \Delta_Y$ such that no component of $\Delta_Y'$
appears in the fixed part of $|k(K_Y+T+\Delta_Y')|$. We put
$\Delta'_Y=\tilde{A}+\sum_ia'_iD_i$.

{\bf Step 4}. Now we deal with intersections $T \cap D_i$ which are
contained in the base locus of $|k(K_Y+T+\Delta_Y')|$. Let $h \colon
\overline{Y} \rightarrow Y$ be the blowup of $T \cap D_i$. Note that
because $T$ is smooth and $T \cap D_i \subset T$ is a divisor the
strict transform of $T$ maps isomorphically to $T$. Let $F \subset
\overline{Y}$ be the exceptional divisor. We write
$$K_{\overline{Y}}+(T+h^*\Delta'_Y-F)=h^*(K_Y+T+\Delta'_Y).$$
Note that $T+h^*\Delta'_Y-F$ is effective, the coefficient of $F$
being $a'_i$.

On $\overline{Y}$ the divisors $T$ and $\tilde{D_i}$ are disjoint.
We need to blowup again along $F \cap \tilde{D_i}$, but this gives
an isomorphism around $T$. We repeat this process; however, we can
continue only finitely many times because
$${\rm ord}_{T\cap F}|k(K_{\overline{Y}}+T+h^*\Delta'_Y-F)|_T\leq {\rm ord}_{T\cap D_i}
|k(K_X+T+\Delta'_Y)|_T-1.$$

{\bf Step 5}. We change notation to denote by $(Y,T+\Delta'_Y)$ the
resulting pair (we emphasize that $T$ hasn't changed starting with
Step 3). We can now apply the extension theorem for
$(Y,T+\Delta'_Y)$. Consider the commutative diagram
\begin{eqnarray*}
\begin{array}{ccc}
H^0(X,k(K_X+S+\Delta)) & \rightarrow & H^0(S,k(K_X+S+\Delta)|_S) \\
\downarrow           &             & \downarrow \\
H^0(Y,k(K_Y+T+\Delta_Y')) & \rightarrow &
H^0(T,k(K_Y+T+\Delta_Y')|_T)
\end{array}
\end{eqnarray*}
The left arrow is an isomorphism and the right arrow is injective by
construction, and the bottom arrow is surjective by the extension
theorem. Hence writing $\Theta_1=\Delta_Y'|_T$, we see that the
component of degree $k$ in $\sR|_S$ is isomorphic to
$H^0(T,k(K_T+\Theta_1))$.

{\bf Step 6}. In order to prove further properties of the restricted
algebra, one needs to combine the previous construction with ``taking
log mobile parts". For every $m \ge 1$ write
$$mk(K_Y+T+\Delta_Y)= M_m + (\mbox{fixed part}).$$
Let $0 \le \Delta_m \le \Delta_Y$ be such that $\Delta_m$ has no
common component with the above fixed part. After possibly replacing
$k$ by a multiple, $\{mk(K_Y+T+\Delta_m)\}$ is an additive sequence.
We now apply Steps 3-5 for each $(Y,T+\Delta_m)$. Note that $T$
remains unchanged. We get models $Y_m \rightarrow Y$ and divisors
$\Theta_m$ on $T$ such that the $k$th truncation of $\sR|_S$ is
isomorphic to
$$\oplus_{m \ge 0} H^0(T,km(K_T+\Theta_m)).$$
Moreover $\Theta_m \le \Theta:=\Delta_Y|_T$ implies $\Theta' :=
\lim_{m \rightarrow \infty} \Theta_m$ exists and $(T,\Theta')$ is
klt. This proves that the restricted algebra is an adjoint algebra.

\chapter{Existence of flips II} \label{corti}
 
This chapter is an exposition of work of Shokurov.

Recall (from Chapter~\ref{mustata})

\begin{defn}
Let $Y \rightarrow Z$ be a projective morphism with $Y$ smooth and $Z$ affine.
An \emph{adjoint algebra} is an algebra of the form 
$$\sR = \oplus_{m \ge 0} H^0(Y,N_m)$$
where $N_m = mk(K_Y +\Delta_m)$, the limit $\Delta := \lim_{m \rightarrow \infty} \Delta_m$ exists, and $(Y,\Delta)$ is klt. 
\end{defn}

\begin{rem}
Note that $\Delta$ can be an $\bR$-divisor.
\end{rem}

In Chapter~\ref{mustata}, it was shown that the ``restricted algebra" is an adjoint algebra.

\begin{defn}
Given an adjoint algebra $\sR(Y,N_{\bullet})$, set $M_i = \Mob N_i$, the mobile part of $N_i$, and $D_i=\frac{1}{i}M_i$.
We say $\sR$ is $a$-\emph{saturated} if there exists a $\bQ$-divisor $F$ on $Y$, $\lceil F \rceil \ge 0$,
such that $\Mob \lceil jD_i+F \rceil \le jD_j$ for all $i \ge j > 0$.
\end{defn}

\begin{rem} In applications, $F$ is always the discrepancy of some klt pair $(X,\Delta)$, $f \colon Y \rightarrow X$.
That is 
$$K_Y=f^*(K_X+\Delta)+F.$$
We sometimes write $F=\bA(X,\Delta)_Y$.
\end{rem}

\begin{ex} An $a$-saturated adjoint algebra on an affine curve is finitely generated.
 
Let $Y=\bC$ and $P=0 \in \bC$. Let $N_i = m_i \cdot 0$, where $m_i+m_j \le m_{i+j}$, and $D_i=\frac{1}{i}m_i \cdot 0 = d_i \cdot 0$.
By assumption $d= \lim_{i \rightarrow \infty} d_i \in \bR$ exists. In this context, $a$-saturation means there exists $b < 1$, $F=-b\cdot 0$, such that
$$\lceil jd_i-b \rceil \le jd_j$$
for all $i \ge j > 0$.

We want to show $d \in \bQ$, and $d=d_j$ for $j$ divisible.
Passing to the limit as $i \rightarrow \infty$ we get
$$\lceil jd-b \rceil \le jd_j.$$
Assume $d \notin \bQ$. Then 
$$\{ \langle jd \rangle \ | \ j \in \bN \} \subset [0,1]$$
is dense (here $\langle \cdot \rangle$ denotes the fractional part).
So, there exists $j$ such that $\langle jd \rangle > b$, and then
$$jd_j \le jd < \lceil jd -b \rceil \le jd_j,$$
a contradiction.
So $d \in \bQ$. The same argument shows that $d_j=d$ if $j \cdot d \in \bZ$.
\end{ex}

\begin{defn} An adjoint algebra $\sR(Y,N_{\bullet})$ is \emph{semiample} if the limit $D=\lim_{i \rightarrow \infty} \frac{1}{i}M_i$
is semiample, where $M_i:=\Mob(N_i)$.
\end{defn}

\begin{rem}
We say that an $\bR$-divisor $D$ on $Y$ is \emph{semiample} if there exists a morphism $f \colon Y \rightarrow W$ with $W$ quasiprojective such that
$D$ is the pullback of an ample $\bR$-divisor on $W$. We say an $\bR$-divisor is \emph{ample} if it is positive on the Kleiman--Mori cone of curves.
\end{rem}

\begin{thm}
If an adjoint algebra $\sR=\sR(Y,N_{\bullet})$ is $a$-saturated and semiample then it is finitely generated.
\end{thm}

The proof is a modification of the one dimensional case, based on the following 

\begin{lem}
Let $Y \rightarrow Z$ be projective with $Y$ smooth and $Z$ affine and normal. Let $D$ be a semiample $\bR$-divisor on $Y$, and assume that $D$ is not a 
$\bQ$-divisor.
Fix $\epsilon > 0$.
There exists a $\bZ$-divisor $M$ and $j > 0$ such that
\begin{enumerate}
\item $M$ is free.
\item $\|jD-M\|_{\sup} < \epsilon$
\item $jD-M$ is not effective.
\end{enumerate}
\end{lem}

\begin{thm}
Assume the MMP in dimension $n$ (precisely, MMP with scaling for klt pairs with $\bR$-coefficients. For the definition of MMP with scaling see Chapter~\ref{kollar}). Let $\sR=\sR(Y,N_{\bullet})$ be an adjoint algebra. Let $N_i=ik(K_Y+\Delta_i)$, 
$\Delta=\lim_{i \rightarrow \infty} \Delta_i$, and assume $K_Y+\Delta$ big. Then there exists a modification $\pi \colon Y' \rightarrow Y$ and $N'_{\bullet}$ such that
$\sR=\sR(Y',N'_{\bullet})$ is semiample.
\end{thm}
\begin{rem}
In the context of flips, the condition $K_Y+\Delta$ big is not an issue.
\end{rem}
The proof of the theorem is based on the following  

\begin{lem}
Let $(X,\Delta)$ be a klt pair, where $X$ is $\bQ$-factorial, $\Delta$ is an $\bR$-divisor, and $K_X+\Delta$ is big. Assume the MMP in dimension $n$.
Let $\Delta \in V \subset \Div_{\bR} X$ be a finite dimensional vector space. There exist $\epsilon > 0$ and finitely many $g \colon X \dashrightarrow W_i$
(birational maps) such that if $D \in V$, $\|D-\Delta\| < \epsilon$, then for some $i$ the pair $(W_i,{g_i}_*D)$ is a log minimal model of $(X,D)$.
\end{lem}
For the proof of the lemma see Chapter~\ref{kollar}.

That the restricted algebra is $a$-saturated can be proved by a straightforward application of Kawamata--Viehweg vanishing as follows.

\begin{thm}
If $-(K_X+S+\Delta)$ is big and nef, then the restricted algebra is $a$-saturated.
\end{thm}
\begin{proof}
Let $(X,S+\Delta)\to Z$ be a pl-flipping contraction, and consider a birational morphism
$f \colon Y \rightarrow X$. Let $T$ denote the strict transform of $S$. Write 
$$K_Y+T+\Delta_Y=f^*(K_X+S+\Delta)+E,$$ 
where $\Delta_Y$ and $E$ are effective (with no common components) and $E$ is exceptional. 
Assume $k(K_X+S+\Delta)$ is integral and Cartier. 
Let $\Delta_m$ be the largest divisor such that $0\le \Delta_m \le \Delta$ and 
$$M_m:=\Mob( mk(K_Y+T+\Delta_m))=\Mob (mk(K_Y+T+\Delta)).$$ 
We may assume $M_m$ is free (for some $f$).
Write $M_m^0=M_m|_T$.
For simplicity we only consider the case $i=j$, i.e., we prove
$$\Mob\lceil M_j^0+F \rceil\le M_j^0,$$ 
where $F=\bA(S,\Delta|_S)_T$ . 

If $\lceil E \rceil$ is an $f$-exceptional divisor, then $\Mob \lceil M_j+E \rceil \le M_j$. 
Choose $E=\bA(X,S+\Delta)_Y+T$, so $E|_T=F$, and consider the short exact sequence
$$0\to \cO_Y (M_j+\bA(X,S+\Delta)_Y)\to \cO_Y (M_j+ E) \to \cO_T(M_j^0+ F) \to 0.$$
By the Kawamata-Viehweg vanishing theorem and our assumption, we know $H^1(\cO_Y (M_j+\bA(X,S+\Delta)_Y))=0$ (because $\bA(X,S+\Delta)_Y=K_Y-f^*(K_X+S+\Delta)$). Now the $a$-saturated condition $\Mob \lceil M_j^0+F \rceil \le M_j^0$ is implied by $\Mob \lceil M_j+E \rceil \le M_j$ and this extension result.

\end{proof}

\chapter{Notes on Birkar-Cascini-Hacon-M\textsuperscript{c}Kernan} \label{kollar}
By the previous 3 lectures, we can start with the:
\medskip

{\bf Assumption:} We  proved the existence of flips in dimension $n$, using  
minimal models in dimension $n-1$.
\medskip

The main result is the following:

\begin{thm}[Minimal models in dimension $n$]
Assume that  $\Delta$ is a big $\r$-divisor,
$(X,\Delta)$ is klt and $K+\Delta$ is pseudo-effective.
Then $(X,\Delta)$ has a minimal model.
\end{thm}

Let us first see some corollaries.

\begin{cor} $\Delta$ is a big $\q$-divisor,
$(X,\Delta)$ is klt and $K+\Delta$ is pseudo-effective.
Then the canonical ring
$$
\sum_{m\geq 0} H^0(X, \o_X(mK_X+\rdown{m\Delta})) \qtq{is finitely generated.}
$$
\end{cor}

\begin{proof} 
Get minimal model for $(X,\Delta)$, then use base point freeness.
\end{proof}

\begin{cor}  $X$ smooth, projective  and $K_X$ is big. Then the canonical ring
$$
\sum_{m\geq 0} H^0(X, \o_X(mK_X)) \qtq{is finitely generated.}
$$
\end{cor}

\begin{proof}
Pick some effective $D\sim mK_X$.
Then $(X,\epsilon D)$ is klt (even terminal) for
$0<\epsilon \ll 1$ and $\epsilon D$ is big.
So $(X,\epsilon D)$ has a minimal model.
It is automatically a
minimal model for $X$.
\end{proof}

\begin{cor} $X$ smooth, projective.
Then the canonical ring
$$
\sum_{m\geq 0} H^0(X, \o_X(mK_X)) \qtq{is finitely generated.}
$$
\end{cor}

\begin{proof} As in Kodaira's canonical bundle formula for
elliptic surfaces, Fujino-Mori reduces the ring
to a general type situation in lower dimension.
\end{proof}

\begin{cor} Let $X$ be a Fano variety. Then the Cox ring
$$
\sum_{D\in \pic(X)} H^0(X, \o_X(D)) \qtq{is finitely generated.}
$$
\end{cor}

\begin{proof} 
See original.
\end{proof}

\begin{cor} If $K+\Delta$ is not pseudo-effective, 
then there exists a birational map $X \dashrightarrow X'$ and a 
Mori fiber space
$X'\to Z'$.
\end{cor}

\begin{proof} Fix $H$ ample and the smallest $c>0$ such that
 $K+\Delta+cH$ is  pseudo-effective. (Note that a priori, 
 $c$ may not be rational.) After MMP, we get $X\map X'$ such that
$K+\Delta+cH$ is nef on $X'$. It cannot be big since then
$K+\Delta+(c-\epsilon)H$ would still be effective.
So base point freeness gives $X'\to Z'$.
\end{proof}

\section{Comparison of 3 MMP's} 

In a minimal model program (or MMP)
we start with a pair $(X,\Delta)$ and the
goal is to  construct a minimal model $(X,\Delta)^{min}$.
(Note that minimal models are usually not unique, so
$(X,\Delta)^{min}$ is a not well defined notational convenience.)

There are 3 ways of doing it.

\subsection{Mori-MMP (libertarian)}  This is the by now classical
approach. 
Pick \emph{any}  extremal ray, contract/flip as
needed. The hope is that eventually we get $(X,\Delta)^{min}$.
This is known only in dimension $\leq 3$
and almost known in dimension 4.

Note that even if $(X,\Delta)$ is known to have a minimal model,
it is not at all clear that every Mori-MMP starting with
$(X,\Delta)$ has to end (and thus yield a  minimal model). 

\subsection{MMP with scaling (dictatorial)}

Fix $H$ and $t_0>0$ such that $K+\Delta+t_0H$ is nef. 

Let $t\to 0$. For a while $K+\Delta+tH$ is nef, but then
we reach 
a critical value $t_1\leq t_0$. 
That is, there exists an extremal ray $R_1$ such that
$R_1\cdot \bigl(K+\Delta+(t_1 -\eta) H\bigr)<0$ 
for $\eta>0$. Contract/flip this $R_1$ and continue.

We will show that MMP with scaling works in any dimension,
provided that $\Delta$ is big 
and $K+\Delta$ is pseudo-effective.

\subsection{Roundabout MMP of BCHM} 

Instead of directly going for a minimal model,
we start by steps that seem to make things more complicated.
Then we aim for the minimal model on a carefully chosen path.
There are 5 stages:

\begin{enumerate}
\item Increase $\Delta$ in a mild manner to  $\Delta^+$.
\item Increase $\Delta^+$ wildly to $\Delta^++M$.
\item Construct $(X,\Delta^++M)^{min}$.
\item Remove the excess $M$ to get $(X,\Delta^+)^{min}$.
\item Prove that $(X,\Delta^+)^{min}$ is also $(X,\Delta)^{min}$.
\end{enumerate}
You can imagine the process as in the picture:
$$
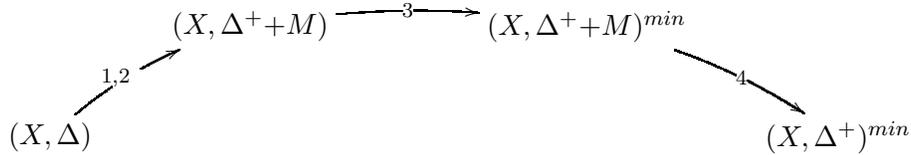
\begin{array}{cc}
\xymatrix{
& (X,\Delta^+{+}M) \ar@/^/[rr]|-{3} &&
(X,\Delta^+{+}M)^{min}  \ar@/^/[rd]|-{4} &\\
(X,\Delta) \ar@/^/[ru]|-{1,2\ \ } &&&&   (X,\Delta^+)^{min}
}\\
\mbox{Bending it like BCHM}
\end{array}
$$
(Note that the ``C'' of Cascini is pronounced
as a ``K''.)
We will show that bending  works in any dimension,
provided that $\Delta$ is big 
and $K+\Delta$ is pseudo effective.

\medskip
(Major side issues). My presentation below ignores
3 important points.
\begin{enumerate}
\item Difference between  klt/dlt/lc. The main results are for
klt pairs but many intermediate steps are needed
for dlt or lc. These are purely technical points
but in the past proofs collapsed on such technicalities. 
\item Relative setting. The induction and applications
sometimes need the relative case: dealing with
morphisms $X\to S$ instead of projective varieties.
The relevant technical issues are well understood. 
\item I do not explain how to prove the
{\it non-vanishing theorem}: If $K+\Delta$ is pseudo-effective
(that is, a limit of effective divisors) then
it is in fact effective. 
\end{enumerate}


\subsection{Spiraling induction}

We  prove 4 results together as follows:
$$
\boxed{\mbox{MMP  with scaling in dim.\ $n-1$}}
$$

$$
\hphantom{\mbox{Section 3}}\Downarrow\mbox{Section 3}
$$

$$
\boxed{\mbox{Termination  with scaling in dim.\ $n$ near $\rdown{\Delta}$}}
$$

$$
\hphantom{\mbox{Section 4}}\Downarrow\mbox{Section 4}
$$

$$
\boxed{\mbox{Existence of $(X,\Delta)^{min}$ in dim.\ $n$}}
$$

$$
\hphantom{\mbox{Section 5}}\Downarrow\mbox{Section 5}
$$

$$
\boxed{\mbox{Finiteness  of $(X,\Delta+\sum t_i D_i)^{min}$ in dim.\ $n$
for $0\leq t_i\leq 1$}}
$$

$$
\hphantom{\mbox{Section 2}}\Downarrow\mbox{Section 2}
$$

$$
\boxed{\mbox{MMP with scaling in dim.\ $n$}}
$$

\bigskip
(Finiteness questions).
 While we expect all 3 versions of the MMP to work
for any $(X,\Delta)$, the above 
finiteness is quite subtle.
Even a smooth surface can contain infinitely many extremal rays,
thus the very first step of the Mori-MMP sometimes
offers an infinite number of choices.

By contrast,  if  $\Delta$ is  big, then
there are only finitely many possible models
reached by   Mori-MMP. This is, however, much stronger than
what is needed for the proof.

It would be very useful
 to pin down what kind of finiteness to expect in general.

\section{MMP with scaling}

Begin with $K+\Delta+t_0 H$ nef, $t_0>0$.

\begin{enumerate}
\item Set $t=t_0$ and decrease it. 
\item We hit a first \emph{critical value} $t_1\leq t_0$.
Here  $K+\Delta+t_1 H$ nef but
$K+\Delta+(t_1 -\eta) H$ is not nef for $\eta>0$. 
\item This means that 
there exists an extremal ray  $R\subset \overline{NE}(X)$
such that 
$$
R\cdot (K+\Delta+t_1 H)=0\qtq{and} R\cdot H>0.
$$
Thus $R\cdot (K+\Delta)<0$ and $R$ is a ``usual'' extremal ray. 
\item Contract/flip  $R$ to get $X_0\map X_1$  and continue.
\end{enumerate}

The problem is that we could  get an infinite sequence
$$
X_1\map X_2\map X_3 \map\cdots .
$$
{\emph Advantage of scaling:}
In the MMP with scaling,
 each $X_i$ is a minimal model for some  $K+\Delta+ t H$. 
So, if we know finiteness of models  as $t$ varies 
then there is no infinite sequence. Thus
we have proved the implication
$$
\begin{array}{c}
\boxed{\mbox{Finiteness  of $(X,\Delta+tH)^{min}$ in dim.\ $n$
for $0\leq t\leq t_0$}}\\
\Downarrow\\
\boxed{\mbox{MMP with scaling in dim.\ $n$}}
\end{array}
$$

\begin{rem}
In the Mori-MMP, the $X_i$ are 
\emph{not} minimal models of anything predictable.
This makes it quite hard to prove termination
as above since we would need to control all possible models
in advance.
\end{rem}

\section{MMP  with scaling  near $\rdown{\Delta}$}

Start with  $(X, S+\Delta)$, $S$ integral. Here
$(X, S+\Delta)$ is assumed dlt and
in all interesting cases $S\neq 0$, so $(X, S+\Delta)$ is not
klt.

We run MMP with scaling to get a series of contractions/flips
$$  \cdots \map
(X_i, S_i+\Delta_i)\stackrel{\phi_i}{\map}
(X_{i+1}, S_{i+1}+\Delta_{i+1}) \map \cdots
$$

Instead of termination, we claim only a much weaker result,
traditionally known as {\it special termination}.

\begin{prop} There are only finitely many $i$ such that
$S_i\cap \exc(\phi_i)\neq \emptyset$.
\end{prop}

\begin{proof} This is all old knowledge, relying on 6 basic
observations.
\begin{enumerate}
\item We can concentrate on 1 component of $S$ and
by a perturbation trick we can assume that $S$ is irreducible.
\item The discrepancy $a(E, X_i, S_i+\Delta_i)$ weakly increases with $i$ and 
strictly increases iff $\cent_{X_i}(E)\subset \exc(\phi_i)$.
\item There are very few $E$ with $a(E, X_i, S_i+\Delta_i)<0$.
\item If $S_i\map S_{i+1}$ creates a {\it new} 
divisor  $F_{i+1}\subset S_{i+1}$, then 
there exists $E_{i+1}$  with 
$a(E_{i+1}, X_{i}, S_{i}+\Delta_{i})<
a(E_{i+1}, X_{i+1}, S_{i+1}+\Delta_{i+1})\leq 0.$
\item Combining these we see that $S_i\map S_{i+1}$ creates
 no new divisors for $i\gg 1$.
\item Only finitely many $S_i\map S_{i+1}$ contracts a divisor, since at
such a step the Picard number of $S_i$ drops.
\end{enumerate}

Thus, for $i\gg 1$, each $S_i\map S_{i+1}$ is an isomorphism
in codimension 1. (Maybe such a map could be called a  \emph{traverse}.)

Therefore,  if  $X_0\map X_1 \cdots $ is an MMP with scaling, then 
$$
S_N\stackrel{\phi_N|_{S_N}}{\map} S_{N+1}\map S_{N+1} \map \cdots
$$
is  also an MMP with scaling for $N\gg 1$,
except that $\phi_i|_{S_i}$ is an isomorphism
when $S_i\cap \exc(\phi_i)\neq \emptyset$.
By induction,  $\phi_i|_{S_i}$ is an isomorphism for $i\gg 1$,
hence $S_i\cap \exc(\phi_i)= \emptyset$.
\end{proof}

Thus we have shown that
$$
\begin{array}{c}
\boxed{\mbox{MMP  with scaling in dim.\ $n-1$}}\\
\Downarrow\\
\boxed{\mbox{Termination  with scaling in dim.\ $n$ near $\rdown{\Delta}$}}
\end{array}
$$

\bigskip

\section{Bending it like BCHM}

This is the hardest part. We
assume termination  with scaling  near $\rdown{\Delta}$
(for many different $X$ and $\Delta$),
and we prove that the roundabout MMP also works.

The proof uses 2 basic lemmas. The first one shows that
under certain conditions, every flip we have to do
involves $\rdown{\Delta}$. The second one shows
how to increase $\rdown{\Delta}$
without changing $(X,\Delta)^{min}$.

\begin{lem}[Scaling to the boundary] Assume that
\begin{enumerate}
\item $K+\Delta\sim cH+F$ for some $c\geq 0$ and $F\geq 0$.
\item $K+\Delta+H$ is nef.
\item $\supp(F)\subset \rdown{\Delta}$.
\end{enumerate}
Then $(X, \Delta+t\cdot H)^{min}$
exists for every $0\leq t\leq 1$.
\end{lem}
\begin{proof} (To accomodate $\r$-divisors, $\sim$ can stand for
$\r$-linear equivalence.)
Start scaling. At the critical value, get a ray $R$
such that $R\cdot H>0$ and $R\cdot (K+\Delta)<0$.
Thus $R\cdot (cH+F)<0$ and $R\cdot F<0$, so
$\locus(R)\subset \supp(F)\subset  \rdown{\Delta}.$

Thus every flip encountered in the scaling MMP
intersects the boundary. Thus we have termination.
\end{proof}

If we start with a klt pair $(X,\Delta)$, then
 $ \rdown{\Delta}=0$. From condition (3) thus $F=0$
and hence $K+\Delta\sim cH$. 
Therefore
$K+\Delta\sim \frac{c}{c+1}(K+\Delta+H)$ is nef
and we have nothing to do. Conclusion:
{\it We need a way to increase $\Delta$ without changing the
minimal model!}

\begin{lem}[Useless divisor lemma]
If $\Delta'\subset \mbox{stable base locus of } (K+\Delta)$, 
then $(X, \Delta+\Delta')^{min}=(X,  \Delta)^{min}$.
\end{lem}

\begin{proof}  Note that 
$$
\begin{array}{ccc}
(\Delta')^{min}&\subset& \mbox{stable base locus of } (K+\Delta)^{min}\\
&& || \\
&& \mbox{stable base locus of } (K+\Delta+\Delta')^{min}
\end{array}
$$
However, $(K+\Delta+\Delta')^{min}$ is base point free, thus
$(\Delta')^{min}=0$, and so
$(X, \Delta+\Delta')^{min}=(X, \Delta)^{min}.$
\end{proof}

\begin{cor} We can always reduce to the case when
the stable base locus of  $K+\Delta^+$ is in
$\rdown{\Delta^+}$.
\end{cor}

\begin{proof} First we can take a  log resolution to
assume that $(X,\Delta)$ has simple normal crossings only.
Write the stable base locus as $\sum_{i=1}^r D_i$
and $\Delta=\sum_{i=1}^r d_iD_i+(\mbox{other divisors})$,
where $d_i=0$ is allowed.
Set $\Delta':=\sum_{i=1}^r (1-d_i)D_i$.
 Then $\Delta^+:=\Delta+\Delta'=\sum_{i=1}^r D_i+(\mbox{other divisors})$
and $\rdown{\Delta+\Delta'}\supset \sum_{i=1}^r D_i$.
\end{proof}

\begin{say}[Bending I:  
$K+\Delta$ is a $\q$-divisor]{\ }

We proceed in 6 steps:

\begin{enumerate}\setcounter{enumi}{-1}
\item Write $K+\Delta \sim rM+F$ where $M$ is mobile, 
irreducible and $F$ is in the stable base locus.
\item Take log resolution.
With $\Delta'$ as above, set $\Delta^+:=\Delta+\Delta'$
and $F^+:=F+\Delta'$. 
Then $K+\Delta^+ \sim rM+F^+$ and
$\supp(F^+)\subset \rdown{\Delta^+}.$
\item Add $M$ to $\Delta^+$ and pick an ample $H$ to get 
	\begin{enumerate}
	\item $K+\Delta^++M\sim 0\cdot H+\bigl((r+1)M+F^+\bigr)$ 
	\item $K+\Delta^++M+H$ is nef, and
	\item $\supp(M+F^+)\subset \rdown{\Delta^++M}$.
	\end{enumerate}
\item Scale by $H$ to get $(X, \Delta^++M)^{min}$. Now we have:
	\begin{enumerate}
	\item $K+\Delta^+\sim rM+F^+$. 
	\item $K+\Delta^++M$ is nef.
	\item $\supp(F^+)\subset \rdown{\Delta^+}$.
	\end{enumerate}
\item Scale by $M$ to get $(X, \Delta^+)^{min}$.
\item By the useless divisor lemma, 
$(X, \Delta^+)^{min}=(X, \Delta)^{min}$.
\end{enumerate}
\end{say}

\begin{say}[Bending II:  
$K+\Delta$ is an $\r$-divisor]{\ }

We follow the same 6 steps, but there are
extra complications.

\begin{enumerate}
\item[(0)]  Write  $K+\Delta \sim r_iM_i+F$, where the $M_i$ are mobile,
irreducible and 
$F$ is in the stable base locus. (This is actually not obvious and
the Oct.\ 2006 version has mistakes.)
\item[(1--3)] goes as before and we  get
\begin{enumerate}
\item $K+\Delta^+\sim \sum r_iM_i+F^+$ 
\item $K+\Delta^++\sum M_i$ is nef, and
\item $\supp(F^+)\subset \rdown{\Delta^+}$.
\end{enumerate}
\item[(4)] Let me first describe 2 attempts that
do not work.

 (First try): Scale  $\sum M_i$. The problem is that
 $\sum r_iM_i\neq c \sum M_i$.

(Second try):  Scale $M_1$. This works, but the support condition 
(c) 
fails at next step when we try to scale $M_2$.

(Third try):  Do not take \emph{all} of $M_1$ away at the first step.

Reorder the indices so that $r_1\leq r_2\leq \cdots \leq  r_k$.
We then  construct inductively minimal models for
$$\bigl(X,\Delta^++\tfrac1{r_j}(r_1M_1+\cdots+r_{j-1}M_{j-1})+
M_j+\cdots +M_k\bigr)$$

We already have the $j=1$ case. Let us see how to do
 $j \rightarrow j+1$.

By assumption
$K+\Delta^++\frac1{r_j}(r_1M_1+\cdots+r_{j-1}M_{j-1})+
(M_j+\cdots +M_k)$ is nef. Move
 $M_j$  from the right sum to the left. Thus now we have:
\begin{enumerate}
\item $K+\Delta^+\sim (r_1M_1+\cdots+r_{j}M_{j})+
(r_{j+1}M_{j+1}+\cdots +r_kM_k+F^+)$ 
\item $K+\Delta^++\frac1{r_j}(r_1M_1+\cdots+r_{j}M_{j})+
(M_{j+1}+\cdots +M_k)$ is nef, and
\item $\supp(r_{j+1}M_{j+1}+\cdots +r_kM_k+F^+)\subset 
\rdown{M_{j+1}+\cdots +M_k+\Delta^+}$.
\end{enumerate}
Scale $r_1M_1+\cdots+r_{j}M_{j}$ with scale factor  $r_j/r_{j+1}$.
This gives the $j+1$ case.

At the end (and after moving $M_k$ into the sum) we have:

\begin{enumerate}
\item $K+\Delta^+\sim \sum r_iM_i+F^+$ 

\item $K+\Delta^++\frac1{r_k}(r_1M_1+\cdots+r_{k}M_{k})$ is nef.
\item $\supp(F^+)\subset \rdown{\Delta^+}$.
\end{enumerate}
Now we can scale by $r_1M_1+\cdots+r_{k}M_{k}$ to get
$(X, \Delta^+)^{min}$.

\item[(5)]  By the useless divisor lemma,
 $(X,\Delta^+)^{min}=(X,\Delta)^{min}$.
\end{enumerate}
\end{say}
Thus we have proved that 
$$
\begin{array}{c}
\boxed{\mbox{Termination  with scaling in dim.\ $n$ near $\rdown{\Delta}$}}\\
\Downarrow\\
\boxed{\mbox{Existence of $(X,\Delta)^{min}$ in dim.\ $n$}}
\end{array}
$$

\section{Finiteness of models}

More generally, we claim that the set of models
$(X, \Delta_w)^{min}$ is finite as
$\Delta_w$ moves in a  compact set of
$\r$-divisors satisfying 3 conditions:
\begin{enumerate}
\item Every $\Delta_w$ is big. Note that being big is not a closed
condition and it would be very good to remove this restriction.
Without it one could get minimal models for non-general type
$(X,\Delta)$ by getting 
$(X,\Delta+\epsilon(\mbox{ample}))^{min}$ and then letting
$\epsilon\to 0$.
\item Every $K+\Delta_w$ is effective. By definition, being
pseudo-effective is a closed condition. It is here that
non-vanishing comes in: it ensures that  being effective
is also
a closed condition.
\item $K+\Delta_w$ is klt. This is preserved by making
$\Delta_w$ smaller, which is what we care about.
\end{enumerate}

By compactness, it is enough to prove
finiteness  locally, that is, finiteness   of the models
 $(X,\Delta'+\sum t_i D_i)^{min}$ 
for $ |t_i|\leq \epsilon$ (depending on $\Delta'$).
(Important point: Even if we only care about $\q$-divisors, we
 need this for $\r$-divisors $\Delta'$!)
\begin{proof} Induction on $r$ for $D_1,\dots, D_r$.
Let $(X^m, \Delta^m):=(X,\Delta')^{min}$. By the base point free theorem,
we have $g:X^m\to X^c$ such that $K_{X^m}+\Delta^m\sim g^*(\mbox{ample})$.
So, for $|t_i|\ll 1$, $g^*(\mbox{ample}) \gg \sum t_i D_i$, 
\emph{except} on the fibers of $g$.
(This is more delicate than it sounds, but not hard
to prove for extremal contractions.)

Thus the MMP to get $(X^m,\Delta^m+\sum t_i D_i)^{min}$ is relative to $X^c$. 
We can switch to working locally over $X^c$,  we thus
 assume that $K_{X^m}+\Delta^m\sim 0$.
So 
$$K_{X^m}+\Delta^m+ c\sum t_i D_i\sim 
c\bigl(K_{X^m}+\Delta^m+ \sum t_i D_i\bigr)$$
Therefore
$(X^m,\Delta^m+ \sum t_i D_i)^{min}=(X^m, \Delta^m+c\sum t_i D_i)^{min}$.
For $(t_1,\dots, t_r)$
choose $c$ such that $\max_i|ct_i|=\epsilon$, that is,
$(ct_1,\dots,ct_r)$ is on a face of $[-\epsilon, \epsilon]^r$.
Ths shows that we get all possible $(X^m,\Delta^m+ \sum t_i D_i)^{min}$
by computing $(X^m,\Delta^m+ \sum t_i D_i)^{min}$
only for those $(t_1,\dots, t_r)$ which are
on a face of the $r$-cube 
$[-\epsilon,\epsilon]^r$. The faces
are $2r$ copies of the $(r{-}1)$-cube.
So we are done by induction on $r$.
The notation $(X,\Delta)^{min}$
helped us  skirt the issue whether there may be
infinitely many minimal models for a given
$(X,\Delta)$. (This happens in the non-general type cases,
even for smooth minimal models.) This is, however, no problem here.
We found one $\q$-factorial model $g:X^m\to X^c$
such that $K_{X^m}$ is numerically $g$-trivial.

Let $D_1,\dots, D_r$ be a basis of the N\'eron-Severi
group. Every other model where $K+\Delta$ is nef lives over $X^c$
and some $\sum t_iH_i$ is $g$-ample on it.
We can thus find every such model as a canonical model
for $(X,\Delta+\epsilon \sum t_iH_i)$ for
$0<\epsilon\ll 1$.
\end{proof}

We have now proved that
$$
\begin{array}{c}
\boxed{\mbox{Existence of $(X,\Delta)^{min}$ in dim.\ $n$}}\\
\Downarrow\\
\boxed{\mbox{Finiteness  of $(X,\Delta+\sum t_i D_i)^{min}$ in dim.\ $n$
for $0\leq t_i\leq 1$}}
\end{array}
$$
\medskip

and the spiraling induction  is complete.

\end{document}